\documentclass[9pt,twosided]{amsart}

\usepackage{hott}

\usepackage{phonetic} %

\newcommand{\shape}{\ensuremath{\mathord{\raisebox{0.5pt}{\text{\rm\esh}}}}}

\newcommand{\bN}{\mathbb N}
\newcommand{\bZ}{\mathbb Z}

\newcommand{\bR}{\mathbb R}

\newcommand{\bS}{\mathbb S}

\newcommand{\Spec}{\ensuremath{\mathrm{Spec}}}
\newcommand{\Spf}{\ensuremath{\mathrm{Spf}}}
\newcommand{\fgkAlg}{\ensuremath{k\mathrm{-Alg}_{\mathrm{fg}}}}
\newcommand{\fgkAlgTop}{\ensuremath{k\mathrm{-Alg}_{\mathrm{fg,top}}}}

\newcommand{\ignore}[1]{}
\newcommand{\et}{\mathsf{et}}

\title{Modal descent}
\author{Felix Cherubini and Egbert Rijke}
\date{\today}

\begin{document}

\maketitle

\begin{abstract}
  Any modality in homotopy type theory gives rise to an orthogonal factorization system of which the left class is stable under pullbacks. We show that there is a second orthogonal factorization system associated to any modality, of which the left class is the class of $\modal$-equivalences and the right class is the class of $\modal$-\'etale maps. This factorization system is called the reflective factorization system of a modality, and we give a precise characterization of the orthogonal factorization systems that arise as the reflective factorization system of a modality. In the special case of the $n$-truncation the reflective factorization system has a simple description: we show that the $n$-\'etale maps are the maps that are right orthogonal to the map $\unit \to \sphere{n+1}$. We use the $\modal$-\'etale maps to prove a modal descent theorem: a map with modal fibers into $\modal X$ is the same thing as a $\modal$-\'etale map into a type $X$. We conclude with an application to real-cohesive homotopy type theory and remarks how $\modal$-étale maps relate to the formally etale maps from algebraic geometry. 
\end{abstract}

\section{Introduction}
In 2011 Urs Schreiber and Mike Shulman introduced Modalities to Homotopy Type Theory,
with the idea to use these extended theories to reason about more specialized $(\infty,1)$-toposes.
One special application they had in mind was to use Homotopy Type Theory to talk about \emph{cohesive} $(\infty,1)$-toposes \cite{ShulmanSchreiber}.
This idea of a cohesive type theory was later developed in \cite{ShulmanRealCohesion} for a special case.
While the results in this article are about one \emph{monadic} modality,
these ideas were relevant for the development of our results and we will discuss possible applications along these lines.
  
Monadic modalities, which we will just call modalities in this article, were defined in \cite[Section 7.7]{UFP}. They were studied extensively in \cite{RijkeSpittersShulman}, where it was shown that any modality gives rise to an orthogonal factorization system of which the left class is stable under pullbacks. Hence we will call this factorization system the \emph{stable} factorization system of a modality. One of our main results is \cref{thm:rfs_orthogonal}, which shows that there is a second orthogonal factorization system that can be obtained from a modality: the \emph{reflective} factorization system. The left maps of the reflective factorization system are the maps that are inverted by the modality and the right maps are those with a cartesian naturality square. In the case where the modality is lex, those left and right classes coincide with the left and right classes of the stable orthogonal factorization system. The reflective factorization system was already used in category theory, e.g. in \cite{cassidy_hebert_kelly_1985},
where the reflector of a reflective subcategory takes the role of the modality. In \cref{thm:rfs} we give a precise characterization of those orthogonal factorization systems that arise as the reflective factorization system of a modality.

We call right maps of the reflective factorization system $\modal$-étale, where $\modal$ is the modality.
This name is inspired by the formally étale maps from algebraic geometry,
which are maps reminiscent of local homeomorphisms in topology.
In topology, ``local'' means that the maps are trivial over some open subset, while \emph{formally} étale maps are trivial on \emph{formal} disks.
In the case of modalities there is a similar notion of $\modal$-disks, and we show that a map is $\modal$-étale maps if and only if it is trivial on $\modal$-disks in a sense made precise in \cref{prp:locally_trivial}.
The relevant definitions from algebraic geometry are included in \cref{subsection:algebraic geometry}
together with proofs that they could be defined analogous to our definitions of $\modal$-étale maps and $\modal$-disks.

Another way in which $\modal$-\'etale maps can be seen as locally trivial is the fact that a map $p:E\to B$ is $\modal$-\'etale if and only if it extends uniquely to a map $\tilde{p}:\tilde{E}\to \modal B$ with $\modal$-modal fibers. This claim, which we establish in \cref{thm:modal_descent}, is the modal descent theorem. 

In \cref{thm:char_n_etale} we use our abstract theory to prove the following characterization of étale maps for the $n$-truncations, for $n\geq -1$:
A map $f:A\to B$ is $n$-\'etale if and only if it is right orthogonal to the base point inclusion $\unit\to\sphere{n+1}$.

In \cref{subsection:toplogical stacks} we show how the modal descent theorem (\cref{thm:modal_descent})
subsumes the classical fundamental theorem of the theory of covering spaces in real-cohesive homotopy type theory.
This also yields a candidate extension of this classical theorem to topological stacks.
In fact, it is essentially in this context that the modal descent theorem was already found and proven by Urs Schreiber as \cite[Proposition 5.2.42]{SchreiberDcct}.
  
We thank Jonas Frey and Mike Shulman for help in understanding factorization systems in a long email discussion in early 2018. This was also the time when the second author wrote his PhD thesis, so some of our results we present there, especially those in \cref{section:modal-descent} and \cref{section:reflective-factorization-system} have already appeared in \cite{rijke-phd}.
Discussions with and remarks of Jonathan Zachhuber, Tobias Columbus, Marcelo Fiore, Steve Awodey, Eric Finster, André Joyal, Mathieu Anel, and Dan Christensen were helpful for this work.
The anonymous reviewers greatly improved the article with their comments and suggestions. 
This material is based upon work supported by the Air Force Office of Scientific Research under award number FA9550-17-1-0326, and through MURI grant FA9550-15-1-0053.

\section{Preliminaries}
We assume that the reader is familiar with the basics of homotopy type theory \cite{UFP} and the basic theory of (idempotent, monadic) modalities, as presented in \cite{RijkeSpittersShulman}. In this preliminary section, we recall the basic concepts from those two sources.

Just as in \cite{UFP}, we write $x=y$ for the type of identifications of $x$ and $y$, provided that both $x$ and $y$ have a common type $X$. Sometimes we call identifications equalities. We write
\begin{equation*}
  \apfunc{f}:(x=y)\to (f(x)=f(y))
\end{equation*}
for the action on identifications of a function $f$. Concatenation of identifications is written in diagrammatic order, i.e., we write $\ct{p}{q}$ for the concatenation of $p:x=y$ and $q:y=z$. The fiber of a map $f:A\to B$ at $b:B$ is defined to be the type
\begin{equation*}
  \fib{f}{b}\defeq \sm{x:A}f(x)=b.
\end{equation*}
Recall that a type $X$ is said to be \define{contractible} if it comes equipped with a term of type
\begin{equation*}
  \iscontr(X)\defeq \sm{x:X}\prd{y:X}x=y.
\end{equation*}
A map is an equivalence if and only if all its fibers are contractible.

We will frequently make use of the concept of \emph{proposition} in homotopy type theory. Propositions are types of which all identity types are contractible, i.e., a type $X$ is said to be a proposition if it comes equipped with a term of type
\begin{equation*}
  \isprop(X)\defeq\prd{x,y:X}\iscontr(x=y).
\end{equation*}
It is important in homotopy type theory to distinguish between properties and structures. A type $P(x)$ indexed by $x:X$ is said to be a \define{property} of $X$ if the type $P(x)$ is a proposition. Otherwise, it is called a \define{structure} on $x$. The type $\mathsf{Prop}$ of all propositions in a universe $\UU$ is defined by
\begin{equation*}
  \mathsf{Prop}\defeq \sm{X:\UU}\isprop(X).
\end{equation*}

We make extensive use of homotopy pullbacks. The most important property we will be relying on is the following theorem:

\begin{thm}\label{thm:pullback}
  Consider a commuting square
  \begin{equation*}
    \begin{tikzcd}
      A \arrow[r,"h"] \arrow[d,swap,"f"] & X \arrow[d,"g"] \\
      B \arrow[r,swap,"i"] & Y
    \end{tikzcd}
  \end{equation*}
  with homotopy $H:i\circ f\htpy g\circ h$. Then the following are equivalent:
  \begin{enumerate}
  \item The square is a pullback square.
  \item For each $b:B$ the induced map on fibers
    \begin{equation*}
      \fib{f}{b} \to \fib{g}{i(b)}
    \end{equation*}
    given by $(a,p)\mapsto (h(a),\ct{H(a)^{-1}}{\ap{i}{p}})$, is an equivalence.
  \end{enumerate}
\end{thm}

For an arbitrary commuting square, the induced map into the pullback is called the \define{gap map}. In other words, the gap map of a commuting square
  \begin{equation*}
    \begin{tikzcd}
      A \arrow[r,"h"] \arrow[d,swap,"f"] & X \arrow[d,"g"] \\
      B \arrow[r,swap,"i"] & Y
    \end{tikzcd}
  \end{equation*}
is the unique map $A\to B\times_Y X$ obtained via the universal property of the pullback. One can show that the fibers of the gap map are equivalent to the fibers of the induced maps on fibers. This observation implies the above theorem.

\cref{thm:pullback} has many nice consequences. We mention two here, both of which can be seen as descent theorems.

\begin{thm}[Descent theorem for $\Sigma$-types]\label{thm:descent-sigma}
  A family of commuting squares
\begin{equation*}
  \begin{tikzcd}
    A_i \arrow[d,swap,"f_i"] \arrow[r] & X \arrow[d,"g"] \\
    B_i \arrow[r] & Y
  \end{tikzcd}
\end{equation*}
indexed by $i:I$ is a family of pullback squares if and only if the induced square
\begin{equation*}
  \begin{tikzcd}
    \sm{i:I}A_i \arrow[d] \arrow[r] & X \arrow[d,"g"] \\
    \sm{i:I}B_i \arrow[r] & Y
  \end{tikzcd}
\end{equation*}
is a pullback square.
\end{thm}

\begin{thm}[Descent theorem for surjective maps]\label{thm:descent-surjective}
  Consider a diagram of the form
  \begin{equation*}
    \begin{tikzcd}
      A \arrow[d] \arrow[r] & X \arrow[d] \arrow[r] & V \arrow[d] \\
      B \arrow[r,swap,"h"] & Y \arrow[r] & W,
    \end{tikzcd}
  \end{equation*}
  in which the left square is a pullback square, and suppose that the map $h:B\to Y$ is surjective. Then the outer rectangle is a pullback if and only if the right square is a pullback.
\end{thm}

The main object of study in this article is a modality, of which the canonical examples are the $n$-truncations. There are many equivalent ways of saying what a modality is \cite{RijkeSpittersShulman}.

\begin{defn}
  A \define{reflective subuniverse} consists of a subuniverse $P:\UU\to\mathsf{Prop}$ equipped with:
  \begin{enumerate}
  \item a \define{modal operator} $\modal :\UU\to\UU$ such that $P(\modal X)$ holds for any $X:\UU$,
  \item a \define{modal unit} $\eta:X\to\modal X$ for each $X:\UU$, that satisfies the universal property of $\modal$-localization: the precomposition function
    \begin{equation*}
      \blank\circ\eta : (\modal X\to Y)\to (X\to Y)
    \end{equation*}
    is an equivalence for every type $Y:\UU$ such that $P(Y)$ holds.
  \end{enumerate}
  Types that satisfy the property $P$ are usually called \define{$\modal$-local}, and we write $\UU_\modal$ for the type of all $\modal$-local types.
\end{defn}

By the universal property of reflective subuniverses, it follows that for every map $f:A\to B$, there is a unique map $\modal f : \modal A \to \modal B$ such that the square
\begin{equation*}
  \begin{tikzcd}
    A \arrow[d,swap,"\eta"] \arrow[r,"f"] & B \arrow[d,"\eta"] \\
    \modal A \arrow[r,dashed,swap,"\modal f"] & \modal B
  \end{tikzcd}
\end{equation*}
commutes. This square is called the \define{$\modal$-naturality square} of $f$.

\begin{prp}\label{prp:modality}
  Given a reflective subuniverse $\modal$, the following two properties are equivalent:
  \begin{enumerate}
  \item For any family $B(x)$ of $\modal$-local types, indexed by $x$ in a $\modal$-local type $A$, the type $\sm{x:A}B(x)$ is also $\modal$-local. We also say that $\modal$ is \define{$\Sigma$-closed}, if this property holds.
  \item For any type $X$, and any family $B:\modal X\to\UU_\modal$ of $\modal$-local types, the precomposition function
    \begin{equation*}
      \blank\circ\eta : \Big(\prd{y:\modal X}B(y)\Big)\to\Big(\prd{x:X}B(\eta(x))\Big)
    \end{equation*}
    is an equivalence. We also say that $\modal$ is \define{uniquely eliminating} if this property holds.
  \end{enumerate}
  If either of these equivalent properties holds, then we say that the reflective subuniverse $\modal$ is a \define{modality}. If $\modal$ is a modality we call the $\modal$-local types \define{$\modal$-modal}.
\end{prp}

It is not the case, however, that any reflective subuniverse is a modality. For example, the subuniverse of types $X$ that are $p$-local in the sense that the precomposition map
\begin{equation*}
  \mathsf{deg}(p) : X^{\sphere{1}}\to X^{\sphere{1}}
\end{equation*}
is an equivalence, where $\mathsf{deg}(p):\sphere{1}\to\sphere{1}$ is the degree $p$-map for some prime $p$, is not a modality \cite{CORS}.

Any modality determines a \emph{stable orthogonal factorization system}, which we recall now.

\begin{defn}
  An \define{orthogonal factorization system} is a pair $(\mathcal{L},\mathcal{R})$ of classes of maps
  \begin{align*}
    \mathcal{L} & : \prd{X,Y:\UU} (X \to Y) \to \mathsf{Prop} \\
    \mathcal{R} & : \prd{X,Y:\UU} (X \to Y) \to \mathsf{Prop} 
  \end{align*}
  such that
  \begin{enumerate}
  \item Both $\mathcal{L}$ and $\mathcal{R}$ contain all equivalences and are closed under composition.
  \item Every map $f:X\to Y$ factors as a left map (i.e.~a map in $\mathcal{L}$) followed by a right map (i.e.~a map in $\mathcal{R}$). More precisely, for every map $f:X\to Y$ there is a type $\mathsf{im}_{(\mathcal{L},\mathcal{R})}(f)$ equipped with maps
    \begin{align*}
      f_{\mathcal{L}} & :X\to \mathsf{im}_{(\mathcal{L},\mathcal{R})}(f)\\
      f_{\mathcal{R}} & :\mathsf{im}_{(\mathcal{L},\mathcal{R})}(f) \to Y
    \end{align*}
    and a homotopy witnessing that the triangle
    \begin{equation*}
      \begin{tikzcd}[column sep=0]
        X \arrow[rr,"f"] \arrow[dr,swap,"f_{\mathcal{L}}"] & & Y \\
        & \mathsf{im}_{(\mathcal{L},\mathcal{R})}(f) \arrow[ur,swap,"f_{\mathcal{R}}"]
      \end{tikzcd}
    \end{equation*}
    commutes.
  \item Every map in the left class is \define{left orthogonal} to every map in the right class (we also say that every map in $\mathcal{R}$ is right orthogonal to every map in $\mathcal{L}$). Following the observations of \cite{AnelBiedermanFinsterJoyal}, this means that for any map $i:A \to B$ in $\mathcal{L}$ and any map $f:X \to Y$ in $\mathcal{R}$, the square
    \begin{equation*}
      \begin{tikzcd}
        X^B \arrow[r] \arrow[d] & Y^B \arrow[d] \\
        X^A \arrow[r] & Y^A
      \end{tikzcd}
    \end{equation*}
    is a pullback square.
  \end{enumerate}
  An orthogonal factorization system is said to be \define{stable} if the left class is stable under pullbacks. That is, for any pullback square
  \begin{equation*}
    \begin{tikzcd}
      A \arrow[r,"h"] \arrow[d,swap,"f"] & X \arrow[d,"g"] \\
      B \arrow[r,swap,"i"] & Y
    \end{tikzcd}
  \end{equation*}
  in which the map $g:X \to Y$ is in $\mathcal{L}$, it is required that $f$ is also in $\mathcal{L}$.
\end{defn}

Recall from \cite{RijkeSpittersShulman} that the stable orthogonal factorization system of a modality is obtained in the following way. First, we say that a map $f:X\to Y$ is \define{$\modal$-modal} if all its fibers are $\modal$-modal types. The class $\mathcal{R}$ is defined to be the class of $\modal$-modal maps. Second, we say that a type $X$ is \define{$\modal$-connected} if $\modal X$ is contractible. Then we say that a map $f:X\to Y$ is \define{$\modal$-connected} if all of its fibers are $\modal$-connected. The class $\mathcal{L}$ is defined to be the class of $\modal$-connected maps. The pair $(\mathcal{L},\mathcal{R})$ is the stable orthogonal factorization system of the modality $\modal$.

Conversely, we can obtain a modality from a stable factorization system, in which a type $X$ is modal if and only if the terminal projection $X\to \unit$ is in $\mathcal{R}$. The modal operator of this modality is defined as
\begin{equation*}
\modal X \defeq \mathsf{im}_{(\mathcal{L},\mathcal{R})}(X\to \unit)
\end{equation*}
and the modal unit is defined to be the left factor $X \to \modal X$ of the map $X\to \unit$. The orthogonality can be used to show that the map $\eta$ defined in this way is indeed uniquely eliminating in the sense of \cref{prp:modality}.

We recall one more useful general fact about modalities.

\begin{thm}
  For any two stable orthogonal factorization systems $(\mathcal{L},\mathcal{R})$ and $(\mathcal{L}',\mathcal{R}')$ the following are equivalent:
\begin{enumerate}
\item Every $(\mathcal{L},\mathcal{R})$-modal type is $(\mathcal{L}',\mathcal{R}')$-modal.
\item The modal units of the modality $(\mathcal{L}',\mathcal{R}')$ are in $\mathcal{L}$.
\end{enumerate}
\end{thm}

Recall that a map is said to be \define{surjective} if all its fibers are merely inhabited. In other words, $f$ is surjective if it is in the left class of the stable factorization system for the $(-1)$-truncation. Therefore we have the following corollary.

\begin{cor}\label{cor:units-surjective}
  For any modality, every proposition is modal if and only if the modal units are surjective.
\end{cor}

\section{\texorpdfstring{$\modal$}{○}-\'etale maps}

\begin{defn}
We say that a map $f:A\to B$ is \define{$\modal$-\'etale}, if the square
\begin{equation*}
\begin{tikzcd}
A \arrow[r,"f"] \arrow[d,swap,"\eta"] & B \arrow[d,"\eta"] \\
\modal A \arrow[r,swap,"\modal f"] & \modal B
\end{tikzcd}
\end{equation*}
is a pullback square. We will write $\isetale(f)$ for this proposition. In the special case where the modality $\modal$ is the $n$-truncation, we will say that a map is \define{$n$-\'etale} if it is $\modal$-\'etale.
\end{defn}

Using the fact that $\modal$ preserves equivalences and composition up to homotopy,
it is immediate from the definition that any equivalence is $\modal$-\'etale, and that the $\modal$-\'etale maps are closed under composition.

\begin{eg}\label{eg:etale_prop}
  We claim that a map $f:A\to B$ is $(-1)$-\'etale if and only if it satisfies the condition
  \begin{equation*}
    A\to \isequiv(f).
  \end{equation*}
  Examples of maps that satisfy this condition include equivalences, maps between propositions, and any map of the form $\emptyt\to B$.

To see that if $f:A\to B$ is $\modal$-\'etale, then $A\to\isequiv(f)$, consider the pullback square
\begin{equation*}
\begin{tikzcd}
A \arrow[r] \arrow[d,swap,"f"] & \brck{A} \arrow[d,"\brck{f}"] \\
B \arrow[r] & \brck{B},
\end{tikzcd}
\end{equation*}
and let $a:A$. Then both $\brck{A}$ and $\brck{B}$ are contractible, so $\brck{f}:\brck{A}\to\brck{B}$ is an equivalence. Since equivalences are stable under pullback it follows that $f$ is an equivalence.

Now suppose that $A\to \isequiv(f)$. Since $\isequiv(f)$ is a proposition, we also have $\brck{A}\to\isequiv(f)$. To see that the gap map
\begin{equation*}
A \to B\times_{\brck{B}}\brck{A}
\end{equation*}
is an equivalence, we will show that its fibers are contractible. Let $b:B$, $x:\brck{A}$ and $p:\bproj{b}=\brck{f}(x)$. Since $\brck{A}\to\isequiv(f)$, it follows that $f$ is an equivalence. Then $\brck{f}$ is also an equivalence, from which it follows that the naturality square is a pullback square. We conclude that the fibers of the gap map are contractible. 
\end{eg}

We saw in the above example that any map between propositions is $\brck{\blank}$-\'etale. This fact generalizes to all modalities.

\begin{lem}\label{lem:etale_modal}
Any map between $\modal$-modal types is $\modal$-\'etale.
\end{lem}

\begin{proof}
Suppose $f:X\to Y$ is a map between $\modal$-modal types. Then the top and bottom maps in the square
\begin{equation*}
\begin{tikzcd}
X \arrow[r] \arrow[d] & \modal X \arrow[d] \\
Y \arrow[r] & \modal Y
\end{tikzcd}
\end{equation*}
are equivalences. Therefore this square is a pullback square, so $f$ is $\modal$-\'etale.
\end{proof}

\begin{rmk}
  If the modality $\modal$ is lex, then it follows from property (viii) in Theorem 3.1 of \cite{RijkeSpittersShulman} that for any $\modal$-modal map $f:A\to B$, the evident map
  \begin{equation*}
    \fib{f}{b}\to\fib{\modal f}{\eta(b)}
  \end{equation*}
  is an equivalence, because it is a $\modal$-connected map between $\modal$-modal types. Therefore we conclude by \cref{thm:pullback} that the square
  \begin{equation*}
    \begin{tikzcd}
      A \arrow[d,swap,"f"] \arrow[r,"\eta"] & \modal A \arrow[d,"\modal f"] \\
      B \arrow[r,swap,"\eta"] & \modal B
    \end{tikzcd}
  \end{equation*}
  is a pullback square. In other words, if the modality $\modal$ is lex, then any $\modal$-modal map is $\modal$-\'etale. The converse holds without assuming that the modality $\modal$ is lex: if $f$ is a base change of $\modal f$, then the fibers of $f$ are $\modal$-modal because the fibers of $\modal f$ are.
\end{rmk}

Our goal in this section is to show that a map is $n$-\'etale, i.e., \'etale for the $n$-truncation, if and only if it is right orthogonal to the point inclusion $\unit\to\sphere{n+1}$. We will use $\modal$-disks in our proof, which we recall from \cite{wellen-thesis}.

\begin{defn}
  Let $\modal$ be a modality, and let $a:A$. The \define{$\modal$-disk} $D^{\modal}(A,a)$ of $A$ at $a$ is defined by
  \begin{equation*}
    D^\modal(A,a)\defeq \sm{x:A}\eta(a)=\eta(x).
  \end{equation*}
  In the special case where $\modal$ is the $n$-truncation, we write $D^n(A,a)$ for the $\trunc{n}{\blank}$-disk at $a$ and if the modality is clear from the context, we allow ourselves to drop ``$\modal$'' from the notation and write just $D(A,a)$.
\end{defn}

Note that the $\modal$-disk fits in a fiber sequence
\begin{equation*}
  \begin{tikzcd}
    D^\modal(A,a) \arrow[r,hook] & A \arrow[r,->>] & \modal(A).
  \end{tikzcd}
\end{equation*}
Moreover, we observe that the $\modal$-disk is $\modal$-connected, since the modal unit $\eta:A\to\modal(A)$ is a $\modal$-connected map. Therefore the $\modal$-disk is also known as the \define{$\modal$-connected cover} of $A$ at $a$.

We also recall the notion of $\modal$-disk bundle from \cite{wellen-thesis}.

\begin{defn}
  For any type $A$, we define the \define{$\modal$-disk bundle}
  \begin{equation*}
    T^\modal A \defeq \sm{x:A}D^\modal(A,x).
  \end{equation*}
\end{defn}

Note that the $\modal$-disk bundle fits in a pullback square
\begin{equation*}
  \begin{tikzcd}
    T^\modal A \arrow[r] \arrow[d] & A \arrow[d] \\
    A \arrow[r] & \modal A.
  \end{tikzcd}
\end{equation*}

Note that $D^\modal$ and $T^\modal$ act functorially: given a map $f:A\to B$ and a point $a:A$, we obtain a map
\begin{equation*}
  D^\modal(f,a) : D^\modal(A,a)\to D^\modal(B,f(a)).
\end{equation*}
This family of maps induces a map $T^\modal f:T^\modal A \to T^\modal B$ for which the square
\begin{equation*}
  \begin{tikzcd}[column sep=large]
    T^\modal A \arrow[r,"T^\modal f"] \arrow[d,swap,"\proj 1"] & T^\modal B \arrow[d,"\proj 1"] \\
    A \arrow[r,swap,"f"] & B
  \end{tikzcd}
\end{equation*}
commutes.

\begin{prp}\label{lem:etale_char}
  Let $\modal$ be a modality and $f:A\to B$ any map, and consider the following two statements:
  \begin{enumerate}
  \item The map $f$ is $\modal$-\'etale.
  \item The square
    \begin{equation*}
      \begin{tikzcd}
        T^\modal A \arrow[r,"T^\modal f"] \arrow[d,swap,"\proj 1"] & T^\modal B \arrow[d,"\proj 1"] \\
        A \arrow[r,swap,"f"] & B
      \end{tikzcd}
    \end{equation*}
    is a pullback square.
  \end{enumerate}
  We have (i) implies (ii). Moreover, if $\eta_A:A\to\modal A$ is surjective, then (ii) implies (i).
\end{prp}

\begin{proof}
  By \cref{thm:pullback} the square
  \begin{equation*}
    \begin{tikzcd}
      A \arrow[r,"f"] \arrow[d,swap,"\eta_A"] & B \arrow[d,"\eta_B"] \\
      \modal A \arrow[r,swap,"\modal f"] & \modal B
    \end{tikzcd}
  \end{equation*}
  is a pullback square if and only if the induced map on fibers
  \begin{equation*}
    \fib{\eta_A}{x}\to \fib{\eta_B}{\modal f(x)}
  \end{equation*}
  is an equivalence for each $x:\modal A$. Thus we see that if $f$ is $\modal$-\'etale, then the map
  \begin{equation*}
    D^\modal(f,x) : D^\modal(A,x)\to D^\modal(B,f(x))
  \end{equation*}
  is an equivalence for each $x:A$. By \cref{thm:pullback} it follows that (ii) is equivalent to the property that each $D^\modal(f,x)$ is an equivalence. This proves that (i) implies (ii). Furthermore, if $\eta_A$ is surjective, then the property that each $D^\modal(f,x)$ is an equivalence is equivalent to the property that each $\fib{\eta_A}{x}\to \fib{\eta_B}{\modal f(x)}$ is an equivalence. This proves that (ii) implies (i) in the case where $\eta_A$ is surjective.
\end{proof}

\begin{eg}\label{rmk:-1etale}
In the special case of $(-1)$-truncation, the characterization of \cref{lem:etale_char} (ii) asserts that a map $f:A\to B$ is $(-1)$-\'etale if and only if the square
\begin{equation*}
\begin{tikzcd}
A\times A \arrow[d,swap,"\proj 1"] \arrow[r,"{f\times f}"] & B\times B \arrow[d,"\proj 1"] \\
A \arrow[r,swap,"f"] & B
\end{tikzcd}
\end{equation*}
is a pullback square. Phrased differently, we see that a map is $(-1)$-\'etale if and only if the square
\begin{equation*}
\begin{tikzcd}
A^{\sphere{0}} \arrow[d,swap,"\mathsf{ev}_\ast"] \arrow[r,"f^{\sphere{0}}"] & B^{\sphere{0}} \arrow[d,"\mathsf{ev}_\ast"] \\
A \arrow[r,swap,"f"] & B
\end{tikzcd}
\end{equation*}
is a pullback square. In other words, $f$ is $(-1)$-\'etale if and only if $f$ is right orthogonal to the base point inclusion $\unit\to\sphere{0}$. 
\end{eg}

\begin{eg}
By \cref{lem:etale_char} and the fact that $\eqv{(\tproj{0}{a}=\tproj{0}{x})}{\brck{a=x}}$, it follows that $f$ is $0$-\'etale if and only if the square
\begin{equation*}
\begin{tikzcd}[column sep=huge]
\sm{a,x:A}\brck{a=x} \arrow[d,swap,"\proj 1"] \arrow[r,"\total{\brck{\apfunc{f}}}"] & \sm{b,y:B}\brck{b=y} \arrow[d,"\proj 1"] \\
A \arrow[r,swap,"f"] & B
\end{tikzcd}
\end{equation*}
is a pullback square. Furthermore, by \cref{thm:pullback} this square is a pullback if and only if the induced map
\begin{equation*}
\Big(\sm{x:A}\brck{a=x}\Big)\to\Big(\sm{y:B}\brck{f(a)=y}\Big)
\end{equation*}
is an equivalence, for each $a:A$.

We note that a map $f:A\to B$ between pointed connected types is an equivalence if and only if it is an embedding, which happens if and only if $f^{\sphere{1}}:A^{\sphere{1}} \to B^{\sphere{1}}$ is an equivalence. We can use this fact to conclude that a map is $0$-connected if and only if the square
\begin{equation*}
  \begin{tikzcd}
    A^{\sphere{1}} \arrow[r] \arrow[d] & B^{\sphere{1}} \arrow[d] \\
    A \arrow[r] & B
  \end{tikzcd}
\end{equation*}
is a pullback square. Therefore we see that a map $f$ is $0$-\'etale if and only if it is right orthogonal to the base point inclusion $\unit\to\sphere{1}$.
\end{eg}

These examples suggest the following theorem.

\begin{thm}
  \label{thm:char_n_etale}
  For any map $f:A\to B$ and any $n\geq -1$, the following are equivalent:
  \begin{enumerate}
  \item The map $f$ is $n$-\'etale.
  \item The map $f$ is right orthogonal to the base point inclusion $\unit\to\sphere{n+1}$.
  \end{enumerate}
\end{thm}

\begin{rmk}
  For $n\jdeq -2$ the statement does not make sense, since there is no base point inclusion $\unit\to\sphere{-1}$. On the other hand, the $(-2)$-\'etale maps are easily characterized: a map is $(-2)$-\'etale if and only if it is an equivalence.
\end{rmk}

\begin{proof}
  The case of $n\jdeq -1$ is already covered in \cref{rmk:-1etale}, so we assume that $n$ is at least $0$. Furthermore, recall that $f$ is right orthogonal to $\unit\to\sphere{n+1}$ if and only if the commuting square
    \begin{equation}\label{eq:orth}
      \begin{tikzcd}
        A^{\sphere{n+1}} \arrow[r] \arrow[d] & B^{\sphere{n+1}} \arrow[d] \\
        A \arrow[r] & B
      \end{tikzcd}
    \end{equation}
    is a pullback square.

  For the forward direction, suppose $f:A\to B$ is $n$-\'etale, and consider the commuting cube
\begin{equation*}
\begin{tikzcd}
&[-1ex] A^{\sphere{n+1}} \arrow[dl] \arrow[d] \arrow[dr] &[2ex] \\
\trunc{n}{A}^{\sphere{n+1}} \arrow[d] & B^{\sphere{n+1}} \arrow[dl] \arrow[dr] & A \arrow[dl,crossing over] \arrow[d] \\
\trunc{n}{B}^{\sphere{n+1}} \arrow[dr] & \trunc{n}{A} \arrow[d] \arrow[from=ul,crossing over] & B \arrow[dl] \\
& \trunc{n}{B}
\end{tikzcd}
\end{equation*}
In this cube the front right square is a pullback square by the assumption that $f$ is $n$-\'etale. The back left square is an exponent of this pullback square, so it is again a pullback. The front left square is a pullback square because its top and bottom map are both equivalences. Therefore we conclude that the back right square is a pullback square, which shows that $f$ is right orthogonal to the map $\unit\to\sphere{n+1}$.

  For the converse, suppose that the square in \cref{eq:orth} is a pullback square. This square is equivalent to the square
  \begin{equation*}
    \begin{tikzcd}
      \sm{x:A}\mathsf{Map}_\ast(\sphere{n},\loopspace{A,x}) \arrow[r] \arrow[d] & \sm{y:B} \mathsf{Map}_\ast(\sphere{n},\loopspace{B,y}) \arrow[d] \\
      A \arrow[r] & B,
    \end{tikzcd}
  \end{equation*}
  so we see that this is a pullback square, and by \cref{thm:pullback} it follows that the map
  \begin{equation*}
    \mathsf{Map}_\ast(\sphere{n},\loopspace{f,x}):\mathsf{Map}_\ast(\sphere{n},\loopspace{A,x}) \to \mathsf{Map}_\ast(\sphere{n},\loopspace{B,f(x)})
  \end{equation*}
  of pointed mapping spaces is an equivalence, for each $x:A$.

  Our goal is to show that $f$ is $n$-\'etale. By \cref{lem:etale_char} it is equivalent to show that the square
  \begin{equation*}
    \begin{tikzcd}
      A\times_{\trunc{n}{A}} A \arrow[r] \arrow[d] & B \times_{\trunc{n}{B}} B \arrow[d] \\
      A \arrow[r] & B
    \end{tikzcd}
  \end{equation*}
  is a pullback square. By \cref{thm:pullback} this is equivalent to showing that the induced map
  \begin{equation*}
    D^{n}(f,x):D^{n}(A,x)\to D^{n}(B,f(x))
  \end{equation*}
  on $\modal$-disks is an equivalence for each $x:A$. We note that the $\modal$-disks are fibers of the unit $\eta:A \to \trunc{n}{A}$, so they are $n$-connected. It follows immediately that the map $D^{n}(f,x)$ is $(n-1)$-connected. Therefore it suffices to show that $D^{n}(f,x)$ is an $(n-1)$-truncated map.

  Recall that a map $\varphi$ between ($0$-)connected types is $(n-1)$-truncated if and only if $\varphi^{\sphere{n+1}}$ is an equivalence. Using our assumption that $n\geq 0$ we know that the $\modal$-disks under consideration are at least connected. Therefore it suffices to show that $(D^{n}(f,x))^{\sphere{n+1}}$ is an equivalence. Now we observe that the square
  \begin{equation*}
    \begin{tikzcd}
      \mathsf{Map}_\ast(\sphere{n+1},D^{n}(A,x)) \arrow[r] \arrow[d] & \mathsf{Map}_\ast(\sphere{n+1},D^{n}(B,f(x))) \arrow[d] \\
      \mathsf{Map}_\ast(\sphere{n+1},(A,x)) \arrow[r] & \mathsf{Map}_\ast(\sphere{n+1},(B,f(x)))
    \end{tikzcd}
  \end{equation*}
  commutes. In this square, the bottom map is an equivalence by the suspension-loop space adjunction, and the fact that $\mathsf{Map}_\ast(\sphere{n},\loopspace{f,x})$ is an equivalence. Therefore it suffices to show that both vertical maps are equivalences, i.e., that any map of the form
  \begin{equation*}
    \proj 1\circ\blank : \mathsf{Map}_\ast(\sphere{n+1},D^{n}(A,x))\to \mathsf{Map}_\ast(\sphere{n+1},(A,x))
  \end{equation*}
  is an equivalence. To see this, we use that $D^n(A,x)$ is equivalent to the type $\sm{y:A}\trunc{n-1}{x=y}$. Therefore it follows that the fiber of the above post-composition map at $(h,\alpha):\mathsf{Map}_\ast(\sphere{n+1},(A,x))$ is equivalent to the type
  \begin{equation*}
    \sm{g:\prd{t:\sphere{n+1}}\trunc{n-1}{h(t)=x}}g(\ast)=\eta(\alpha).
  \end{equation*}
  Here $\alpha$ is the identification $h(\ast)=x$ witnessing that $h$ is a base-point preserving map. However, since $g$ is a dependent function from the $(n+1)$-sphere into a family of $(n-1)$-types, it follows by the dependent universal property of $\sphere{n+1}$ that the type above is equivalent to the type
  \begin{equation*}
    \sm{\beta:\trunc{n-1}{h(\ast)=x}}\beta=\eta(\alpha),
  \end{equation*}
  which is clearly contractible. Therefore we see that the post-composition map $\proj 1\circ\blank$ has contractible fibers, so we conclude that it is an equivalence.
\end{proof}

The proof of \cref{thm:char_n_etale} uses the suspension-loop space adjunction, so it doesn't seem to be directly generalizable to arbitrary accessible modalities. For instance, it would be interesting to know whether a \'etale maps for the nullification modality at an arbitrary pointed type can be characterized in a similar way.

\section{Locally trivial maps}

  In this section we consider a map $f:A\to B$, and write
  \begin{align*}
    F_y & \defeq \fib{\modal f}{y} \\
    D_y & \defeq \fib{\eta}{y}
  \end{align*}
  for any $y:\modal B$. The type $D_y$ can be thought of as a $\modal$-disk, except that it is not centered at a point in $B$.

  Recall that $f$ can be seen as a fibration, of which the fiber at $b:B$ is the type $\fib{f}{b}$. We will show that the condition of being $\modal$-\'etale is related to the condition that the fibration $f$ is trivial on the types $D_y$. We define this condition more precisely as follows.
  
\begin{defn}
 We say that $f$ is \define{$\modal$-locally trivial} if for each $y:\modal B$ there is a map $\varphi_y:F_y\times D_y \to A$ such that the cube
  
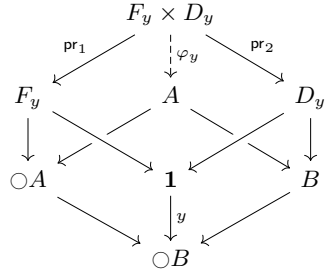
\begin{equation*}
    \begin{tikzcd}
      & F_y\times D_y \arrow[dl,swap,"\proj 1"] \arrow[d,dashed,"\varphi_y"] \arrow[dr,"\proj 2"] \\
      F_y \arrow[d] & A \arrow[dl] \arrow[dr] & D_y \arrow[d] \arrow[dl,crossing over] \\
      \modal A \arrow[dr] & \unit \arrow[from=ul,crossing over] \arrow[d,"y"] & B \arrow[dl] \\
      & \modal B
    \end{tikzcd}
  \end{equation*}
  commutes, and the back-right square is a pullback square.
\end{defn}

By the assumption that the square
\begin{equation*}
  \begin{tikzcd}
    F_y\times D_y \arrow[d,swap,"\proj 2"] \arrow[r,"\varphi_y"] & A \arrow[d,"f"] \\
    D_y \arrow[r,swap,"i"] & B
  \end{tikzcd}
\end{equation*}
is a pullback square, we see that a $\modal$-locally trivial map is a map that becomes a trivial fibration when it is restricted to a $\modal$-disk. Indeed, with \cref{thm:pullback} we obtain from this pullback square a family of equivalences
\begin{equation*}
  F_y \simeq \fib{f}{i(z)}
\end{equation*}
indexed by $z:D_y$, where $i:D_y\to B$ is the fiber inclusion of the (unpointed) fiber sequence $D_y\hookrightarrow B\twoheadrightarrow\modal B$. The commutativity of the cube implies that the map $\varphi_y$ is uniquely determined, as we will soon see.

\begin{prp}\label{prp:locally_trivial}
  The following are equivalent:
  \begin{enumerate}
  \item The map $f$ is $\modal$-\'etale.
  \item The map $f$ is $\modal$-locally trivial.
  \end{enumerate}
\end{prp}

\begin{proof}
  Suppose that $f$ is $\modal$-\'etale, and for arbitrary $y:\modal B$ consider the cube
    
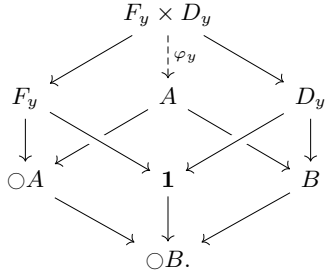
\begin{equation*}
    \begin{tikzcd}
      & F_y\times D_y \arrow[dl] \arrow[d,dashed,"\varphi_y"] \arrow[dr] & \phantom{\modal A}\\
      F_y \arrow[d] & A \arrow[dl] \arrow[dr] & D_y \arrow[d] \arrow[dl,crossing over] \\
      \modal A \arrow[dr] & \unit \arrow[from=ul,crossing over] \arrow[d] & B \arrow[dl] \\
      & \modal B.
    \end{tikzcd}
  \end{equation*}
  In this cube, the map $\varphi_y$ is the unique map such that the cube commutes, obtained from the assumption that the bottom square is a pullback square. Now observe that the bottom, front-left, front-right, and top squares are all pullback squares. Therefore it follows immediately that the remaining squares are pullback squares. Hence $f$ is $\modal$-locally trivial.

  Now assume that $f$ is $\modal$-locally trivial, and consider the commuting cube
  \begin{equation*}
    \begin{tikzcd}[column sep=tiny]
      & \sm{y:\modal B}F_y\times D_y \arrow[dl] \arrow[d] \arrow[dr] & \phantom{\modal A} \\
      \sm{y:\modal B}F_y \arrow[d] & A \arrow[dl] \arrow[dr] & \sm{y:\modal B}D_y \arrow[dl,crossing over] \arrow[d] \\
      \modal A \arrow[dr] & \sm{y:\modal B}\unit \arrow[from=ul,crossing over] \arrow[d] & B \arrow[dl] \\
      & \modal B
    \end{tikzcd}
  \end{equation*}
  Using the descent theorem of $\Sigma$-types (\cref{thm:descent-sigma}) and the assumption that $f$ is $\modal$-locally trivial, we see that the back-right square is a pullback square. We also note that the vertical maps on the left, right and in the front are equivalences. Moreover, we observe that the top square and the front-left square are pullback squares. Therefore it follows that the rectangle 
  \begin{equation*}
    \begin{tikzcd}
      \sm{y:\modal B} F_y\times D_y \arrow[r] \arrow[d] & A \arrow[d,"f"] \arrow[r] & \modal A \arrow[d,"\modal f"] \\
      \sm{y:\modal B}D_y \arrow[r] & B \arrow[r] & \modal B 
    \end{tikzcd}
  \end{equation*}
  consisting of the back-right square and the bottom square in the cube, is a pullback square. Since the map $\sm{y:\modal B}D_y\to B$ is an equivalence, and in particular surjective, we use the descent theorem for surjective maps (\cref{thm:descent-surjective}) to conclude that the square on the right is a pullback square, i.e., that $f$ is $\modal$-\'etale.
\end{proof}

\begin{cor}
  Being $\modal$-locally trivial is a property.
\end{cor}

\begin{proof}
  Since a map is $\modal$-\'etale whenever it is $\modal$-locally trivial, it follows that the type of maps $\varphi_y:F_y\times D_y\to A$ such that the cube commutes is contractible, whenever $f$ is $\modal$-locally trivial.
\end{proof}

\section{Modal descent}\label{section:modal-descent}

\begin{prp}\label{thm:etale_flattening}
Consider a pullback square
\begin{equation*}
\begin{tikzcd}
E' \arrow[d,swap,"{p'}"] \arrow[r,"g"] & E \arrow[d,"p"] \\
B' \arrow[r,swap,"f"] & B
\end{tikzcd}
\end{equation*}
in which we assume that $E$ and $B$ are modal types. Then the square
  \begin{equation*}
    \begin{tikzcd}
      \modal E' \arrow[r,"\tilde{g}"] \arrow[d,swap,"{\modal p'}"] & E \arrow[d,"p"] \\
      \modal B' \arrow[r,swap,"\tilde{f}"] & B,
    \end{tikzcd}
  \end{equation*}
  where $\tilde{f}$ and $\tilde{g}$ are the unique extensions of $f$ and $g$ along the modal units of $B'$ and $E'$, is also a pullback square.
\end{prp}

\begin{proof}
Consider the diagram
\begin{equation*}
\begin{tikzcd}
E' \arrow[r,dashed,"h"] \arrow[d,swap,"{p'}"] & \modal B'\times_{B} E \arrow[r,"\proj 2"] \arrow[d,swap,"\proj 1"] & E \arrow[d,"p"] \\
B' \arrow[r,swap,"\modalunit"] & \modal B' \arrow[r,swap,"\tilde{f}"] & B
\end{tikzcd}
\end{equation*}
In this diagram, the square on the right is a pullback by definition, and the outer rectangle is a pullback by assumption, so the square on the left is also a pullback. Therefore the map $h:E'\to \modal B'\times_B E$ is $\modal$-connected. Moreover, since the modal types are closed under pullbacks it follows that $\modal B'\times_B E$ is modal. Therefore we obtain a commuting diagram of the form
\begin{equation*}
  \begin{tikzcd}[row sep=small]
    \modal E' \arrow[dd] \arrow[rr,dashed,"\tilde{h}"] \arrow[dr] & & \modal B'\times_B E \arrow[dd] \arrow[dl] \\
    \phantom{\modal B'\times_B E} & E \\
    \modal B' \arrow[rr] \arrow[dr] & & \modal B' \arrow[dl] \\
    & B. \arrow[from=uu,crossing over]
  \end{tikzcd}
\end{equation*}
The map $\tilde{h}$ is the unique extension of $h$ along $\eta:E'\to \modal E'$. Note that $\tilde{h}$ is an equivalence, since it extends a $\modal$-connected map. The bottom map in the back square is also an equivalence. Therefore it follows, that the square on the left is equivalent to the square on the right, which is a pullback square. Hence the claim follows.
\end{proof}

\begin{cor}\label{cor:etale_lex}
Consider a pullback square
\begin{equation*}
\begin{tikzcd}
E' \arrow[d,swap,"{p'}"] \arrow[r] & E \arrow[d,"p"] \\
B' \arrow[r] & B,
\end{tikzcd}
\end{equation*}
where $p$ is assumed to be $\modal$-\'etale. We make two claims:
\begin{enumerate}
\item The square
\begin{equation*}
\begin{tikzcd}
\modal E' \arrow[d,swap,"{\modal p'}"] \arrow[r] & \modal E \arrow[d,"\modal p"] \\
\modal B' \arrow[r] & \modal B,
\end{tikzcd}
\end{equation*}
is again a pullback square.
\item The map $p'$ is $\modal$-\'etale.
\end{enumerate}
\end{cor}

\begin{proof}
  Consider the commuting cube
  \begin{equation*}
    \begin{tikzcd}
      &[-.5ex] E' \arrow[dr] \arrow[d] \arrow[dl] \\
      \modal E' \arrow[d] & B' \arrow[dl] \arrow[dr] & E \arrow[dl,crossing over] \arrow[d] \\
      \modal B' \arrow[dr] & \modal E \arrow[d] \arrow[from=ul,crossing over] & B \arrow[dl] \\
      & \modal B.
    \end{tikzcd}
  \end{equation*}
  Since $f$ is assumed to be $\modal$-\'etale, the front-right square is a pullback square. Moreover, the back-right square is also a pullback square by assumption. Therefore the front-left square is a pullback by \cref{thm:etale_flattening}, so the first claim follows. Moreover, we conclude that the back-left square is a pullback, so the second claim follows.
\end{proof}

\begin{defn}
  Let $X$ be a type in a universe $\UU$, and define the type
  \begin{equation*}
    \et/X\defeq \sm{Y:\UU}{g:Y\to X}\isetale(g)
  \end{equation*}
\end{defn}

Now we note that for any map $f:A\to \modal X$ with a $\modal$-modal domain, the pullback of $f$ along $\eta:X\to \modal X$
\begin{equation*}
\begin{tikzcd}
X\times_{\modal X}A \arrow[d,swap,"\proj 1"] \arrow[r,"\proj 2"] & A \arrow[d,"f"] \\
X \arrow[r,swap,"\modalunit"] & \modal X.
\end{tikzcd}
\end{equation*}
is a $\modal$-\'etale map by \cref{lem:etale_modal,cor:etale_lex}. Thus we obtain an operation
\begin{equation*}
  \eta^\ast :\Big(\sm{A:\UU_\modal}A\to\modal X\Big)\to\et/X.
\end{equation*}

The following is a descent theorem for $\modal$-\'etale maps.

\begin{thm}[Modal descent]\label{thm:modal_descent}
For any modality $\modal$, and any type $X$, the operation
\begin{equation*}
\eta^\ast:\Big(\sm{A:\UU_\modal}A\to\modal X\Big)\to\et/X
\end{equation*}
is an equivalence.
\end{thm}

\begin{proof}
If $g:Y\to X$ is $\modal$-\'etale, then the square
\begin{equation*}
\begin{tikzcd}
Y \arrow[d,swap,"g"] \arrow[r,"\modalunit"] & \modal Y \arrow[d,"\modal g"] \\
X \arrow[r,swap,"\modalunit"] & \modal X
\end{tikzcd}
\end{equation*}
is a pullback square. Therefore we see that the map $\modal g : \modal Y \to \modal X$ is in the fiber of $\eta^\ast$ at $g : Y\to X$. 

It remains to show that for any map $f:A\to\modal X$ with modal domain, there is an equivalence $\eqv{A}{\modal (X\times_{\modal X} A)}$ such that the triangle
\begin{equation*}
\begin{tikzcd}[column sep=0]
A \arrow[dr,swap,"f"] \arrow[rr,"\eqvsym"] & & \modal (X\times_{\modal X} A) \arrow[dl,"\modal(\eta^\ast(f))"] \\
\phantom{\modal (X\times_{\modal X} A)} & \modal X
\end{tikzcd}
\end{equation*}
commutes. To see this, note that both $f\circ \proj 2$ and $\modal(\eta^\ast(f))\circ \modalunit$ factor the same map as a $\modal$-connected map followed by a modal map, so the claim follows from uniqueness of factorizations.
\end{proof}

\begin{cor}
Suppose $P:X\to\UU$ is a family of types such that the projection map $\proj 1:\big(\sm{x:X}P(x)\big)\to X$ is $\modal$-\'etale. Then $P(x)$ is $\modal$-modal for each $x:X$, and the map $P:X\to\UU_\modal$ has a unique extension
\begin{equation*}
\begin{tikzcd}
X \arrow[d,swap,"\modalunit"] \arrow[r,"P"] & \UU_\modal. \\
\modal X \arrow[ur,densely dotted,swap,"\tilde{P}"] 
\end{tikzcd}
\end{equation*}
It follows that the commuting square
\begin{equation*}
\begin{tikzcd}
\sm{x:X}P(x) \arrow[d,swap,"\proj 1"] \arrow[r] & \sm{t:\modal X}\tilde{P}(t) \arrow[d,"\proj 1"] \\
X \arrow[r,swap,"\modalunit"] & \modal X
\end{tikzcd}
\end{equation*}
is a pullback square. In particular the top map is $\modal$-connected, so this square is in fact a $\modal$-naturality square.
\end{cor}

\begin{proof}
  Since the square is a pullback, and the bottom map is $\modal$-connected, it follows that the top map is $\modal$-connected. However, the codomain of the top map is $\modal$-modal, so it follows that the square is equivalent to the $\modal$-naturality square
  \begin{equation*}
    \begin{tikzcd}
      \sm{x:X}P(x) \arrow[d,swap,"\proj 1"] \arrow[r] & \modal\Big(\sm{x:X}P(x)\Big) \arrow[d,"\modal \proj 1"] \\
X \arrow[r,swap,"\modalunit"] & \modal X.
    \end{tikzcd}
  \end{equation*}
\end{proof}

\section{$\modal$-equivalences}

The class of $\modal$-equivalences was introduced by \cite{CORS} in the more general case of reflective subuniverses. We will use them in this section to derive some generalizations of the results in the previous section.

\begin{defn}
We say that a map $f:A\to B$ is an \define{$\modal$-equivalence} if $\modal f:\modal A\to \modal B$ is an equivalence.
\end{defn}

\begin{rmk}
The difference between the notions of $\modal$-equivalences and $\modal$-connected maps is best explained by an example. In the case of $n$-truncation, the $n$-equivalences are precisely the maps that induce isomorphisms on the first $n$ homotopy groups. The $n$-connected maps are the maps that induce isomorphisms on the first $n$ homotopy groups, and moreover induce an epimorphism on the $(n+1)$-st homotopy group. 

We also note that the $n$-equivalences are not stable under pullbacks, whereas the $n$-connected maps are. Consider for instance the pullback square
\begin{equation*}
\begin{tikzcd}
\loopspace {\sphere{n+1}} \arrow[r] \arrow[d] & \unit \arrow[d] \\
\unit\arrow[r] & \sphere{n+1}
\end{tikzcd}
\end{equation*}
Here the map on the right is an $n$-equivalence, since $\sphere{n+1}$ is $n$-connected. However, the map on the left is not an $n$-equivalence, since the $n$-th homotopy group of $\loopspace{\sphere{n+1}}$ is not trivial: it is the $(n+1)$-st homotopy group of $\sphere{n+1}$, which is $\Z$.
\end{rmk}
We recall from \cite{CORS} the following facts about $\modal$-equivalences:
\begin{prp}
  \label{lem:3for2_mequiv}\label{cor:mequiv_mconn}
  \begin{enumerate}
  \item The $\modal$-equivalences satisfy the 3-for-2 property.
  \item A map $f:A\to B$ is a $\modal$-equivalence if and only if for every $\modal$-modal type $X$, the precomposition map
    \begin{equation*}
      \precomp{f} : (B \to X) \to (A \to X)
    \end{equation*}
    is an equivalence.
  \item
    Every $\modal$-connected map is a $\modal$-equivalence.
  \end{enumerate}
\end{prp}

We learned about the following generalization of \cref{thm:modal_descent} in a discussion with Anel, Awodey, Joyal, and Shulman: The factorization system of $\modal$-equivalences and $\modal$-\'etale maps is an orthogonal factorization system that satisfies the property that the right class descends along maps in the left class:

\begin{thm}
  For any $\modal$-equivalence $f:A\to B$, the pullback operation
  \begin{equation*}
    f^\ast : \et/B\to\et/A
  \end{equation*}
  is an equivalence.
\end{thm}

\begin{proof}
  Given a $\modal$-equivalence $f:A\to B$, consider the commuting square
  \begin{equation*}
    \begin{tikzcd}
      \et/(\modal B) \arrow[r,"\modal f^\ast"] \arrow[d,"\eta^\ast"] & \et/(\modal A) \arrow[d,"\eta^\ast"] \\
      \et/B \arrow[r,"f^\ast"] & \et/A
    \end{tikzcd}
  \end{equation*}
  By the modal descent theorem (\cref{thm:modal_descent}) it follows that the maps $\eta^\ast$ are equivalences. Furthermore, the map $\modal f$ is assumed to be an equivalence. Therefore it follows that $f^\ast$ is an equivalence.
\end{proof}

Next, we show that $\modal$-equivalences are stable under base change by $\modal$-\'etale maps.

\begin{prp}\label{prop:LR-cartesian-converse}
  Consider a pullback square
  \begin{equation*}
    \begin{tikzcd}
      E' \arrow[d,swap,"{p'}"] \arrow[r,"g"] & E \arrow[d,"p"] \\
      B' \arrow[r,swap,"f"] & B
    \end{tikzcd}
  \end{equation*}
  in which $p$ and $p'$ are $\modal$-\'etale and $f$ is a $\modal$-equivalence. Then the map $g$ is a $\modal$-equivalence.
\end{prp}

\begin{proof}
  Consider the commuting cube
  \begin{equation*}
    \begin{tikzcd}
      \phantom{\modal B} & E' \arrow[dl] \arrow[d] \arrow[dr] & \phantom{\modal B} \\
      \modal E' \arrow[d] & B' \arrow[dl] \arrow[dr] & E \arrow[d] \arrow[dl,crossing over] \\
      \modal B' \arrow[dr] & \modal E \arrow[from=ul,crossing over] \arrow[d] & B \arrow[dl] \\
      & \modal B
    \end{tikzcd}
  \end{equation*}
  In this cube, the back-left, back-right, and front-right squares are pullback squares by assumption. Therefore it follows by \cref{thm:etale_flattening} that the front-left square is a pullback. However, in this square the map $\modal B'\to \modal B$ is an equivalence, so we conclude that the map $\modal E'\to \modal E$ is an equivalence. In other words, the map $g:E'\to E$ is a $\modal$-equivalence.
\end{proof}

  The following theorem is a descent-type result of the kind of \cref{thm:descent-sigma,thm:descent-surjective}.

  \begin{thm}
    Consider a diagram of the form
    \begin{equation*}
      \begin{tikzcd}
        E'' \arrow[d] \arrow[r] & E' \arrow[d] \arrow[r] & E \arrow[d] \\
        B'' \arrow[r,swap,"h"] & B' \arrow[r] & B
      \end{tikzcd}
    \end{equation*}
    in which $h$ is a $\modal$-equivalence, the vertical maps in the right square are $\modal$-\'etale, and the left square is a pullback square. Then the following are equivalent:
    \begin{enumerate}
    \item The outer rectangle is a pullback square.
    \item The square on the right is a pullback square.
    \end{enumerate}
  \end{thm}

  \begin{proof}
    We have that (ii) implies (i) by the pasting lemma for pullbacks, so it suffices to show that (i) implies (ii). Consider the diagram
    \begin{equation*}
      \begin{tikzcd}
        E'' \arrow[rr] \arrow[dd] \arrow[dr] & \phantom{\modal E''} & E' \arrow[rr] \arrow[dd] \arrow[dr] & \phantom{\modal E''} & E \arrow[dd] \arrow[dr] & \phantom{\modal E''} \\
        \phantom{\modal E''} & \modal E'' \arrow[rr,crossing over] & \phantom{\modal E''} & \modal E' \arrow[rr,crossing over] & \phantom{\modal E''} & \modal E \arrow[dd] \\
        B'' \arrow[dr] \arrow[rr] & & B' \arrow[dr] \arrow[rr] & & B \arrow[dr] \\
        & \modal B'' \arrow[from=uu,crossing over] \arrow[rr,swap,"\simeq"] & & \modal B' \arrow[from=uu,crossing over] \arrow[rr] & & \modal B.
      \end{tikzcd}
    \end{equation*}
    In this diagram, the three vertical $\modal$-naturality squares are all pullback squares, because the vertical maps $E'\to B'$ and $E\to B$ are assumed to be $\modal$-\'etale, and the vertical map $E''\to B''$ is $\modal$-\'etale by \cref{cor:etale_lex}. Furthermore, the back-left square and the back rectangle are assumed to be pullback squares. By \cref{thm:etale_flattening} it follows that the front-left square and the front rectangle are pullback squares. Furthermore, the top map in the front-left square is an equivalence. Therefore we see that the front-right square, which is equivalent to the front rectangle, is a pullback square. Using the pullback squares on the sides of the right cube, we conclude that the back-right square is a pullback square.
  \end{proof}

  \section{Reflective factorization systems}\label{section:reflective-factorization-system}

  We will now define the reflective factorization system of a modality, of which the right class is the class of $\modal$-\'etale maps.

\begin{defn}
The \define{reflective factorization system} associated to a modality $\modal$ consists of the $\modal$-equivalences as the left class, and the $\modal$-\'etale maps as the right class.
\end{defn}

\begin{thm}\label{thm:rfs_orthogonal}
  The pair $(\mathcal{L},\mathcal{R})$, where $\mathcal{L}$ is the class of $\modal$-equivalences, and $\mathcal{R}$ is the class of $\modal$-\'etale maps, is an orthogonal factorization system.
\end{thm}

\begin{proof}
  First we show that every map factors as a $\modal$-equivalence followed by a $\modal$-\'etale map. Consider a map $f:A\to B$, and the diagram
\begin{equation*}
\begin{tikzcd}
A \arrow[ddr,bend right=15,swap,"f"] \arrow[drr,bend left=15,"\modalunit"] \arrow[dr,"\mathsf{gap}" description] \\
& B\times_{\modal B} \modal A \arrow[d,swap,"\proj 1"] \arrow[r,"\proj 2"] & \modal A \arrow[d,"\modal f"] \\
& B \arrow[r,swap,"\modalunit"] & \modal B.
\end{tikzcd}
\end{equation*}
Then $\proj 1:B\times_{\modal B} \modal A\to B$ is a pullback of a map between modal types, so it is $\modal$-\'etale by \cref{cor:etale_lex}. Furthermore, the map $\proj 2:B\times_{\modal B}\modal A\to \modal A$ is a pullback of a $\modal$-connected map, so it is $\modal$-connected. It follows from \cref{cor:mequiv_mconn} that it is a $\modal$-equivalence. Since the modal unit $\modalunit :A\to\modal A$ is also $\modal$-connected, and therefore a $\modal$-equivalence, we obtain by the 3-for-2 property of $\modal$-equivalences established in \cref{lem:3for2_mequiv} that the gap map is also a $\modal$-equivalence.
  
It remains to show that for every $\modal$-equivalence $i:A\to B$, and every $\modal$-\'etale map $f:X\to Y$, the square
\begin{equation*}
\begin{tikzcd}
X^B \arrow[r] \arrow[d] & Y^B \arrow[d] \\
X^A \arrow[r] & Y^A
\end{tikzcd}
\end{equation*}
is a pullback square. Consider the commuting cube
\begin{equation*}
\begin{tikzcd}
&[-1ex] X^B \arrow[dl] \arrow[d] \arrow[dr] \\
(\modal X)^B \arrow[d] & X^A \arrow[dl] \arrow[dr] & Y^B \arrow[d] \arrow[dl,crossing over] \\
(\modal X)^A \arrow[dr] & (\modal Y)^B \arrow[from=ul,crossing over] \arrow[d] & Y^A \arrow[dl] \\
& (\modal Y)^A
\end{tikzcd}
\end{equation*}
In this cube the top and bottom squares are pullback squares by the assumption that $f$ is $\modal$-\'etale and the fact that exponents of pullback squares are again pullback squares. Furthermore, the square in the front left is pullback, because the two vertical maps are equivalences by the assumption that $i:A\to B$ is a $\modal$-equivalence. Therefore we conclude that the square in the back right is also a pullback square, as desired.
\end{proof}

The reflective factorization system of a modality enjoys several properties. We highlight two of them, which turn out to characterize the orthogonal factorization systems that arise as the reflective factorization system of a modality. Note that the following proposition is a converse to \cref{prop:LR-cartesian-converse}.

\begin{prp}\label{prop:LR-cartesian}
  Any commuting square of the form
  \begin{equation*}
    \begin{tikzcd}
      E' \arrow[d,swap,"{p'}"] \arrow[r,"g"] & E \arrow[d,"p"] \\
      B' \arrow[r,swap,"f"] & B
    \end{tikzcd}
  \end{equation*}
  in which $p$ and $p'$ are $\modal$-\'etale and $f$ and $g$ are $\modal$-equivalences is a pullback square.
\end{prp}

\begin{proof}
  Consider the cube
  \begin{equation*}
    \begin{tikzcd}
      & E' \arrow[dl] \arrow[d] \arrow[dr] & \phantom{\modal E'} \\
      \modal E' \arrow[d] & B' \arrow[dl] \arrow[dr] & E \arrow[d] \arrow[dl,crossing over] \\
      \modal B' \arrow[dr] & \modal E \arrow[from=ul,crossing over] \arrow[d] & B \arrow[dl] \\
      & \modal B & \phantom{\modal B'}
    \end{tikzcd}
  \end{equation*}
  In this cube the back-left square and the front-right square are pullback squares by the assumption that $p'$ and $p$ are $\modal$-\'etale. Moreover, the maps $\modal B'\to\modal B$ and $\modal E'\to\modal E$ are equivalences by the assumption that $f$ and $g$ are $\modal$-equivalences. Therefore it follows that the front-left square is a pullback square. We conclude that the back-right square is a pullback square.
\end{proof}

\begin{prp}\label{prop:R-stable}
  Suppose $p_i:E_i\to B_i$ is a $\modal$-\'etale map for each $i:I$, where $I$ is assumed to be a $\modal$-modal type. Then the induced map on total spaces
  \begin{equation*}
    \total{p} : \sm{i:I}E_i\to \sm{i:I}B_i
  \end{equation*}
  is also $\modal$-\'etale.
\end{prp}

\begin{proof}
  Since pullback squares are preserved by $\Sigma$, it follows from our assumption that the square
  \begin{equation*}
    \begin{tikzcd}
      \sm{i:I}E_i \arrow[r] \arrow[d] & \sm{i:I}\modal E_i \arrow[d] \\
      \sm{i:I}B_i \arrow[r] & \sm{i:I}\modal B_i
    \end{tikzcd}
  \end{equation*}
  is a pullback square. The vertical map on the right is a map between $\modal$-modal types by the assumption that $I$ is modal. Therefore it follows that the map on the left is $\modal$-\'etale.
\end{proof}

Now we show that if an orthogonal factorization system satisfies the conditions described in \cref{prop:LR-cartesian,prop:R-stable}, then it is the reflective factorization system of a modality.

\begin{thm}\label{thm:rfs}
  Suppose $(\mathcal{L},\mathcal{R})$ is an orthogonal factorization system satisfying the following two properties:
  \begin{enumerate}
  \item Any commuting square of the form
    \begin{equation*}
      \begin{tikzcd}
        E' \arrow[d,swap,"{p'}"] \arrow[r,"g"] & E \arrow[d,"p"] \\
        B' \arrow[r,swap,"f"] & B
      \end{tikzcd}
    \end{equation*}
    in which $p',p\in\mathcal{R}$ and $f,g\in\mathcal{L}$ is a pullback square.
  \item For any family of $\mathcal{R}$-maps
    \begin{equation*}
      p_i:E_i\to B_i
    \end{equation*}
    indexed by a type $I$ such that the terminal projection $I\to\unit$ is in $\mathcal{R}$, the induced map on total spaces
    \begin{equation*}
      \total{p}:\sm{i:I}E_i\to\sm{i:I}B_i
    \end{equation*}
    is also in $\mathcal{R}$. 
  \end{enumerate}
  Then the orthogonal factorization system $(\mathcal{L},\mathcal{R})$ is the reflective factorization system of a modality. We say that $(\mathcal{L},\mathcal{R})$ \define{is a reflective factorization system} if it satisfies the two properties above.
\end{thm}

\begin{proof}
  The subuniverse of modal types is defined to be the subuniverse of types $X$ such that the terminal projection $X\to \unit$ is in $\mathcal{R}$. The modal operator $\modal$ is defined by the $(\mathcal{L},\mathcal{R})$-factorization of the terminal projection:
  \begin{equation*}
    \begin{tikzcd}
      X \arrow[r,"\in\mathcal{L}"] & \modal X \arrow[r,"\in\mathcal{R}"] & \unit.
    \end{tikzcd}
  \end{equation*}
  We first show that this is a reflective subuniverse. Thus, we have to show that for any $\modal$-modal type $Y$, the precomposition function
  \begin{equation*}
    (\modal X\to Y)\to (X\to Y)
  \end{equation*}
  is an equivalence. This follows from orthogonality, since the square
  \begin{equation*}
    \begin{tikzcd}
      Y^{\modal X} \arrow[r] \arrow[d] & Y^X \arrow[d] \\
      \unit^{\modal X} \arrow[r] & \unit^X
    \end{tikzcd}
  \end{equation*}
  is a pullback square if $Y\to\unit$ is in $\mathcal{R}$.

  Next, we show that the reflective subuniverse $\modal$ is $\Sigma$-closed, which is one of the equivalent conditions on a reflective subuniverse to be a modality. Consider a type $X$ such that the terminal projection $X\to\unit$ is in $\mathcal{R}$, and consider a type family $P$ over $X$ such that the terminal projection $P(x) \to \unit$ is in $\mathcal{R}$ for each $x:X$. Then it follows by assumption (ii) that the map
  \begin{equation*}
    \Big(\sm{x:X}P(x)\Big)\to\Big(\sm{x:X}\unit\Big)
  \end{equation*}
  is in $\mathcal{R}$. Thus we see that the composite
  \begin{equation*}
    \begin{tikzcd}
      \sm{x:X}P(x) \arrow[r] & X \arrow[r] & \unit
    \end{tikzcd}
  \end{equation*}
  is in $\mathcal{R}$, which shows that the reflective subuniverse $\modal$ is $\Sigma$-closed. We conclude that it is a modality.

  It remains to show that a map is in $\mathcal{R}$ if and only if it is $\modal$-\'etale. To see this, consider the diagram
  \begin{equation*}
    \begin{tikzcd}
      E \arrow[d,swap,"p"] \arrow[r] & \modal E \arrow[d,"\modal p"] \arrow[r] & \unit \arrow[d] \\
      B \arrow[r] & \modal B \arrow[r] & \unit,
    \end{tikzcd}
  \end{equation*}
  where $p$ is assumed to be in $\mathcal{R}$. The top and bottom maps in the left square are $\mathcal{L}$-maps. Moreover, all the maps in the right square are $\mathcal{R}$-maps. Hence the left square is a pullback by assumption (i).
\end{proof}

Recall from \cite[Section 1]{RijkeSpittersShulman} that there are four equivalent ways of saying what a modality is:
\begin{enumerate}
\item A higher modality.
\item A uniquely eliminating modality.
\item A $\Sigma$-closed reflective subuniverse.
\item A stable orthogonal factorization system.
\end{enumerate}
In other words, the \emph{type} of higher modalities is equivalent to the \emph{type} of uniquely eliminating modalities, and so on. Each of these equivalences preserves the underlying subuniverse of modal types. We can now add a fifth structure to this list:
\begin{enumerate}
  \setcounter{enumi}{4}
\item A reflective factorization system.
\end{enumerate}
Note, however, that this does \emph{not} mean that an orthogonal factorization system is stable if and only if it is reflective. The reflective and stable orthogonal factorization systems of a modality coincide if and only if the modality is lex.

\begin{thm}
  The type of modalities is equivalent to the type of reflective factorization systems.
\end{thm}

\begin{proof}
  In the proof of \cref{thm:rfs} we showed that the right class $\mathcal{R}$ of a reflective factorization system is precisely the class of \'etale maps for the underlying modality. In other words, a reflective factorization system is completely determined by its modal types, so the claim follows.
\end{proof}

\section{Applications in real-cohesive homotopy type theory}
\label{subsection:toplogical stacks}

In \cite[Section 8]{ShulmanRealCohesion} Mike Shulman introduces real-cohesive homotopy type theory.
This type theory is a candidate for an internal language for some specific \emph{cohesive} ($\infty$,1)-toposes.
The term ``cohesion'' refers to a higher analog 
of Lawvere's axiomatic cohesion \cite{Lawvere07} developed by Urs Schreiber \cite{SchreiberDcct}.

In this section, we will assume all the rules of Shulman's real-cohesive homotopy type theory which he also assumes in his article.
Additionally, we will assume Shulman's Axiom ``$\mathrm{R}\flat$''.
We will use univalence without mention and, as Shulman does, we will assume propositional resizing.
From now on, we will refer to this type theory as real-cohesion.
Following Shulman's notation, we will write  ``$\bR$'' for the type of \emph{Dedekind reals}, which will be small by propositional resizing.

In real-cohesion, types can have \emph{both} topological structure and homotopical structure.
We can probe the topological structure of some type $X$ by mapping $\bR$ into $X$,
i.e. by looking at \emph{topological paths} $\gamma:\bR\to X$.
Since having \emph{two different} notions of ``paths'' would be confusing,
we decided to follow the terminology of \cite{ShulmanRealCohesion} in this article and call the elements of identity types identifications or equalities.

We will briefly recall some facts and definitions from real-cohesion.
There is an unfortunate name-clash, since 0-truncated types are sometimes called ``discrete''.
The following definition is about \emph{topological} discreteness and is a priori not related to truncation levels.
\begin{defn}
  A type $X$ is \emph{discrete} if and only if the map
  \[ x\mapsto (y\mapsto x) : X \to (\bR \to X) \]
  is an equivalence.
\end{defn}
Note that in real-cohesion, discreteness is defined without any reference to $\bR$.
The definition in real-cohesion just uses the rules of this type theory and that it can be replaced with the definition above,
is exactly the statement of the Axiom ``$\mathrm{R}\flat$'' (see \cite[Section 8]{ShulmanRealCohesion}).
The Dedekind reals are 0-truncated and turn out to be not discrete.
We import the fact, that the following types are discrete:
\[ \emptyt \quad \unit \quad \bN \quad \bZ \quad \sphere{1} \]
where $\emptyt$ is the empty type,
$\unit$ is the unit-type and $\sphere{1}$ is the higher inductive type representing the homotopy type of the $1$-sphere.
Note that the latter is denoted with $\bS^1$ in \cite{UFP}, which we will use for the topological $1$-sphere:
\begin{defn}
  Let $\bS^{1}$ denote the topological sphere given by
  \[ \bS^{1}:\equiv\left\{ (x,y)\in\bR^{2}\left\vert x^2+y^2=1\right.\right\}.\]
\end{defn}

The discrete types are the modal types of a modality that can be constructed as \emph{nullification} at $\bR$,
which is a general construction defined in \cite[Section 2.3]{RijkeSpittersShulman}.
\begin{defn}
  Let $\shape$ be a modality called ``shape'' given by nullification at $\bR$.
\end{defn}
By construction as a nullification at $\bR$, shape will nullify $\bR$, which means $\shape \bR=\unit$. 
In general, shape may be thought of as mapping topological spaces to their \emph{homotopy types}.
Using the rules of real-cohesion, Shulman computes $\shape \bS^1 =\sphere{1}$.

We will denote the modal unit of $\shape$ with $\modalunit_X\colon X \to \shape X$, for a type $X$.

Let $\ast:\bS^1$ be a fixed point on the topological circle.
For $\bS^{1}$ the $\shape$-disk
\[ D (\bS^1,\ast) \equiv \sum_{x:\bS^{1}}\modalunit_{\bS^{1}}(x)=\modalunit(\ast) \]
turns out to be the universal cover of $\bS^{1}$.
But this works only for spaces with trivial higher homotopy groups. 
For the construction of the universal cover of an arbitrary type, this has to be adjusted:
\[ \widetilde{X}:\equiv \sum_{x\colon X} \trunc{0}{\modalunit_X(x)=\ast}. \]
Note that this type would again be a fiber of a unit, if we had a modality that takes the shape and $1$-truncates it.
It is not clear to us, if the simple definition $\shape_1:\equiv\trunc{1}{\blank}\circ\shape$ works.
One way to make it work, would be to show that truncations of discrete types are again discrete types.
But it is not known by the authors if this is true and it seems to be an open problem
\footnote{We have to thank one of our anonymous reviewers for pointing out the problem with the simple definition and the solution we use below.}.
In \cite[Theorem 3.28]{RijkeSpittersShulman} it is shown that for any two accessible modalities,
there is a modality such that its modal types are the meet of the modal types of the two modalities.
So we can make the following definition:
\begin{defn}
Let the $1$-shape, $\shape_1$ be the modality given as the meet of the accessible modalities $\trunc{1}{\blank}$ and $\shape$ with $\shape_1$.
\end{defn}
Then, also from \cite{RijkeSpittersShulman} we know, that a type is $\shape_1$-modal if and only if it is discrete and $1$-truncated.
In \cite[Theorem 6.21]{ShulmanRealCohesion} it is shown that crisply discrete types have discrete $n$-truncations.
So for crisp types $X$, we have
\[ \trunc{1}{\shape X} \simeq \shape_1 X\text{.} \]
If $X$ is a pointed type, the fundamental group with respect to its topological structure can be defined as the loop space
\[ \Omega \left(\shape_1 X\right). \]
Since we will use this notion only once, we will not denote it with $\pi_1$ to limited confusion with common definitions of homotopy type theory.

With $\shape_1$ and its unit $\eta$, covering spaces and the universal cover are easy to define:
\begin{defn}
  \begin{enumerate}
  \item A map $f:X\to Y$ is called a \emph{covering space}, if it is $\shape_1$-étale and $0$-truncated.
  \item Let $X$ be a pointed type. Then
    \[ \widetilde{X}:\equiv D^{\shape_1}(X,\ast)\equiv\sum_{x:X}\modalunit(x)=\modalunit(\ast) \]
    is the \emph{universal cover} of $X$.
  \end{enumerate}

\end{defn}
The following observations justify these names:
\begin{rmk}
  Let $X$ be any pointed type.
  \begin{enumerate}
  \item The projection from the universal cover $\widetilde{X}$ is a covering space.
  \item We have $\shape_1 \widetilde{X} = \unit$.
  \item Let $f:Y\to X$ be a covering space. Then we have the following lifting property:
    A map $g:Z\to X$ lifts uniquely to $Y$, if $\shape_1 g$ lifts to $\shape_1 Y$ along $\shape_1 f$.
  \item Let $f:Y\to X$ be a covering space and $f$ a pointed map. Then there is a unique map $\widetilde{X}\to Y$ such that
    \begin{center}
      \begin{tikzcd}
        \widetilde{X}\arrow[rr]\arrow[dr] & & Y\arrow[dl, "f"] \\
        & X & \\
      \end{tikzcd}
    \end{center}
    commutes.
  \end{enumerate}
\end{rmk}
\begin{proof}
  \begin{enumerate}
  \item By lemma \ref{cor:etale_lex}. 
  \item Applying theorem \ref{thm:modal_descent} to $\unit\to \shape_1 X$ yields this result directly.
  \item This is the universal property of the pullback square from the definition of $\modal$-étale maps.
  \item This is an application of (iii).
  \end{enumerate}
\end{proof}

The construction of the covering spaces corresponding to a subgroup $H\subseteq \pi_1(X)$ for a path connected $X$,
can be done by applying the delooping construction of \cite{LicataFinster} to the inclusion map of $H$ to get a map $Bi:BH\to\shape_1(X)$ and pulling $Bi$ back along $\modalunit$. In other words, we use that any subgroup $H\subseteq\pi_1(X)$ can be represented by an action of $\pi_1(X)$ on a discrete $0$-type
\footnote{}
and therefore a map $BH\to \shape_1X$, with discrete $BH$.

To get the full correspondence for some general type $X$ of actions of the fundamental groupoid of $X$ on sets and covering spaces over $X$,
we can apply theorem \ref{thm:modal_descent} to $\shape_1$ to get:

\begin{thm}
  \begin{enumerate}
  \item Let $X$ be a type. Then the type of $\shape_1$-étale maps into $X$ and the type of $\shape_1$-modal dependent types over $\shape_1 X$ are equivalent.
  \item The type of covering spaces and the type of maps $\shape_1X\to \mathcal U_{\shape_0}$ are equivalent.
  \item Let $X$ be pointed and such that $\shape_1 X$ is connected. Then $\shape_1X\to \mathcal U_{\shape_0}$ is the type of actions of the fundamental group $\Omega(\shape_1 X)$ on discrete $0$-types and this type is again equivalent to covering spaces of $X$.
  \end{enumerate}
\end{thm}
\begin{proof}
  \begin{enumerate}
  \item This is just lemma \ref{thm:modal_descent} applied to $\shape_1$.
  \item By pullback pasting and surjectivity of $\modalunit_X$, fibers of $1$-covering spaces over $X$ are always equivalent to values of the corresponding morphism $\shape_1 X\to \mathcal U_{\shape_1}$ and vice versa.
  \item The equivalence holds by (ii).
    That maps of the form $\rho:\shape_1 X\to \mathcal U_{\shape_0}$ are actions of the loop space of $\shape_1 X$ on the value $\rho(\ast)$ is a consequence of the \emph{homotopical} covering theory of \cite[Section 3.1]{favonia-thesis} and \cite[Section 7.1]{ulrik-egbert-floris-groups}.
  \end{enumerate}
\end{proof}

Similar generalizations of the classical topological correspondence are known on the classical side
for example for cohesive $\infty$-stacks \cite[Section 5.2.7]{SchreiberDcct} or \cite{dmr-2covers}.
The introduction of the latter also gives more details on the history of the subject,
in particular concerning definitions of covering spaces topological and differentiable stacks.

\section*{Appendix: Analogous constructions in algebraic geometry}
\label{subsection:algebraic geometry}
The results we present here are certainly known to experts in algebraic geometry,
but we were not able to find a suitable reference in the literature.
The purpose of this section is to present the name-giving analogs of $\modal$-étale maps and $\modal$-disks from algebraic geometry.

Noetherian schemes are spaces of interest in algebraic geometry. There is a notion of formally étale maps between such spaces.
One purpose of this section is to show that such maps are characterized in very much the same way as $\modal$-étale maps:
We will define a pointed endofunctor on a category containing Noetherian schemes
such that the maps with cartesian naturality squares are precisely the formally étale maps.
Also in this section, we will show that formal disks or formal neighbourhoods of points can be constructed analogous to $\modal$-disks.

The functor $\Im$ we will define below, arises most naturally in algebraic geometry but can also be adapted to differential geometry.
How an analogous functor can be used in differential geometry is described and studied intensively in \cite{SyntheticPDEs}.

In the following, $k$ will always be a field and all rings and algebras are assumed to be commutative and equipped with a unit for multiplication.
We denote the category of finitely generated algebras over $k$ with $k\mathrm{-Alg}_{\mathrm{fg}}$.
That means, that any $A\in k\mathrm{-Alg}_{\mathrm{fg}}$ is a quotient $A=k[x_1,\dots,x_n]/(f_1,\dots,f_m)$.
These algebras may contain nilpotent elements, i.e. elements $x\in A$, such that $x\neq 0$, but $x^n=0$ for some $n\in\bN$.
Nilpotent elements will be important for our constructions, since they represent infinitesimals.
This can roughly be explained by the analogy
that the elements of the algebras are to be thought of as generalized coordinate functions
and the nilpotents represent coordinates that are so (infinitesimally) small, that some power is actually zero. 

We use the notation $\Spec(A)$ for the Hom-functor $k\mathrm{-Alg}_{\mathrm{fg}}(A,\blank)$ from $k\mathrm{-Alg}^{\mathrm{op}}_{\mathrm{fg}}$ to the category of sets.
These functors represent so called \emph{affine Noetherian $k$-schemes} and they form the basic building blocks of spaces called Noetherian $k$-schemes
(see \cite[Chapter II]{hartshorne} for more on schemes).
We will use no property of Noetherian $k$-schemes here,
except that we can descent to affine Noetherian $k$-schemes .

For any $X\in\mathrm{Psh}(k\mathrm{-Alg}^\mathrm{op}_{\mathrm{fg}})$,
the functor $\Im X$ defined pointwise by
\[ (\Im X)(A):\equiv X(A/\sqrt{0})\]
is again a functor from $X\in\mathrm{Psh}(k\mathrm{-Alg}^\mathrm{op}_{\mathrm{fg}})$ to the sets.
So $\Im$ is an endofunctor on the presheaf category $\mathrm{Psh}(k\mathrm{-Alg}^\mathrm{op}_{\mathrm{fg}})$.

The name $\modal$-étale is an adaption of the name ``formally étale'' for general modalities.
The name ``formally étale'' was used in \cite{wellen-thesis}, which reused the name from \cite{SyntheticPDEs}.
The original definition of formally étale maps is from algebraic geometry.
The definition of formally étale maps in \cite[§ 17]{GrothendieckDieudonne} states that a comparison map to a pullback should be an isomorphism,
which is equivalent to the unique lifting condition in the following definition:
\begin{defn}
  A morphisms of schemes $\varphi:X\to Y$ is \emph{formally étale},
  if for all rings $R$ and all nilpotent ideals $N$ in $R$ all squares
  \begin{center}
    \begin{tikzcd}
      \mathrm{Spec(R/N)}\arrow[r]\arrow[d] & X\arrow[d, "f"] \\
      \mathrm{Spec(R)}\arrow[r]\arrow[ru, dashed, "\exists!"] & Y
    \end{tikzcd}
  \end{center}
  have a unique lift like indicated in the diagram.
\end{defn}
We will now make a remark which explains how the formally étale maps from algebraic geometry relate to our notion of $\modal$-étale.
The presented fact and its proof is mostly a repetition of a proof from \cite[Section 4.4]{wellen-thesis}.
We use the fact, that the Noetherian $k$-schemes are embedded in the category $X\in\mathrm{Psh}(k\mathrm{-Alg}^\mathrm{op}_{\mathrm{fg}})$.
\begin{rmk}
  A morphism of $f:X\to Y$ Noetherian schemes $X,Y$ is formally étale if and only if the naturality square
  \begin{center}
    \begin{tikzcd}
      X\arrow[r]\arrow[d, "f", swap] & \Im X\arrow[d, "\Im f"] \\
      Y\arrow[r] & \Im Y
    \end{tikzcd}
  \end{center}
  is a pullback square.
\end{rmk}
\begin{proof}
  Let $X$ and $Y$ be Noetherian $k$-schemes. 
  We will first show that a morphism of schemes $f:X\to Y$ is formally étale,
  if and only if for all $A\in\fgkAlg$ all squares
  \begin{center}
    \begin{tikzcd}
      \mathrm{Spec(A/\sqrt{0})}\arrow[r]\arrow[d] & X\arrow[d, "f"] \\
      \mathrm{Spec(A)}\arrow[r]\arrow[ru, dashed, "\exists!"] & Y
    \end{tikzcd}
  \end{center}
  have a unique lift like indicated in the diagram. Let us call this property (1).
  Since $\sqrt{0}$ is always nilpotent in a Noetherian ring, (1) is implied if $f$ is formally étale.
  
  The property formally étale is known to be local in the source \cite[§ 17.1.6]{GrothendieckDieudonne}, so we can assume $X$ and $Y$ to be affine.
  For affine $X=\mathrm{Spec(S)}$ and $Y=\mathrm{Spec(S)}$, all squares
  \begin{center}
    \begin{tikzcd}
      \mathrm{Spec(A/N)}\arrow[r]\arrow[d] & X\arrow[d, "f"] \\
      \mathrm{Spec(A)}\arrow[r, "\mathrm{Spec}(\varphi)"] & Y
    \end{tikzcd}
  \end{center}
  factor as 
  \begin{center}
    \begin{tikzcd}
      \mathrm{Spec(A/N)}\arrow[r]\arrow[d] & \mathrm{Spec}(\mathrm{im}(\varphi)/(\mathrm{im}(\varphi)\cap N))\arrow[r]\arrow[d] & X\arrow[d, "f"] \\
      \mathrm{Spec(A)}\arrow[r] & \mathrm{Spec}(\mathrm{im}(\varphi))\arrow[r] & Y
    \end{tikzcd}
  \end{center}
  That means we can assume $A$ to be Noetherian, if $X$ and $Y$ are Noetherian for the sake of checking if $f$ is formally étale.
  So let us assume (1) holds. Let $A$ be Noetherian and let us construct a unique lift in
  \begin{center}
    \begin{tikzcd}
      \mathrm{Spec(A/N)}\arrow[r]\arrow[d] & X\arrow[d, "f"] \\
      \mathrm{Spec(A)}\arrow[r] & Y
    \end{tikzcd}
  \end{center}
  We extend the square by reducing $R$ or equivalently $R/N$:
  \begin{center}
    \begin{tikzcd}
      \mathrm{Spec}(A/\sqrt{0})\arrow[rd]\arrow[d] & \\
      \mathrm{Spec(A/N)}\arrow[r]\arrow[d] & X\arrow[d, "f"] \\
      \mathrm{Spec(A)}\arrow[r] & Y
    \end{tikzcd}
  \end{center}
  There are two ways to view the boundary of this diagram as a square, so we can apply (1) in two different ways.
  One application tells us, that the map $\mathrm{Spec}(A/N)\to X$ is the unique one making the diagram commute.
  The second application yields a unique lift:
  \begin{center}
    \begin{tikzcd}
      \mathrm{Spec}(A/\sqrt{0})\arrow[rd]\arrow[d] & \\
      \mathrm{Spec(A/N)}\arrow[d] & X\arrow[d, "f"] \\
      \mathrm{Spec(A)}\arrow[r]\arrow[ur, dashed] & Y
    \end{tikzcd}
  \end{center}
  which is also a lift in the original square by the uniqueness of the map $\mathrm{Spec}(A/N)\to X$.
  This proves that (1) implies that $f$ is formally étale.

  So what remains to be shown is that (1) is equivalent to 
  \begin{center}
    \begin{tikzcd}
      X\arrow[r]\arrow[d, "f", swap] & \Im X\arrow[d, "\Im f"] \\
      Y\arrow[r] & \Im Y
    \end{tikzcd}
  \end{center}
  being a pullback. This is true if and only if it is true pointwise, i.e. for all $k$-algebras $A$,
  the squares 
  \begin{center}
    \begin{tikzcd}
      X(A)\arrow[r]\arrow[d, "f_A", swap] & \Im X(A)\arrow[d, "\Im f_A"] & = X(A/\sqrt{0}) \\
      Y(A)\arrow[r] & \Im Y(A) & = Y(A/\sqrt{0})
    \end{tikzcd}
  \end{center}
  have to be pullback squares. But this is just (1) by Yoneda.
\end{proof}

\ignore{ ... to much work now to dig up a reference for this one ...
Early versions of the characterization of $\Im$-étale maps in \cref{lem:etale_char} were developed with :

Let $X$ and $Y$ be Noetherian $k$-schemes such that the natural map $X\to \Im X$ is pointwise surjective,
then $f:X\to Y$ is formally étale if and only if it induces an isomorphism on the formal disk at $x$ for all $x\in X$.

The condition that $\eta_X:X\to \Im X$ is surjective is implied by formal smoothness which is a reformulation of pointwise surjectivity of $\eta_X$.
The sheaves $D^\Im(X,x)$ turn out to be something well known in algebraic geometry, which we will explore now.
}

In algebraic geometry, there is the concept of the \emph{formal completion} of a closed subspace (see \cite[p.194]{hartshorne} or \cite[10.8]{ega}).
Roughly, the formal completion of a subspace may be thought of as the subspace together with
all points from the surrounding space which are infinitesimally close to the subspace.
In the affine case, where a closed subspace of $\Spec(A)$ is given by an ideal $I\subseteq A$, we can construct a topological ring $\hat{A}$
as the limit of the sequence of quotients by powers of $I$ with discrete topology:
\begin{center}
    \begin{tikzcd}
      \dots\arrow[r] & A/I^3\arrow[r] & A/I^2\arrow[r] & A/I.
    \end{tikzcd}
\end{center}
Let us write $\Delta(A)$ for $A\in\fgkAlg$ with the discrete topology.
The completion yields a functor:
\[ \Spf(\hat{A}) := \fgkAlgTop(\hat{A},\Delta(\_))\]
\begin{rmk}
  Let $X$ be a Noetherian $k$-scheme and $D^\Im(X,x)$ be given as the pullback:
  \begin{center}
    \begin{tikzcd}
      D^\Im(X,x)\arrow[d]\arrow[r] & \unit \arrow[d, "\eta_X\circ x"] \\
      X\arrow[r] & \Im X
    \end{tikzcd}
  \end{center}
  Then $D^\Im(X,x)$ is the formal neighborhood of $x$ in $X$. 
\end{rmk}
\begin{proof}
  Since formal completions are defined by descending to affine schemes, we can assume $X=\Spec(A)$ with $A\in\fgkAlg$.
  Then $x:\unit\to X$ can be rewritten as $x:\Spec(k)\to\Spec(A)$ and thus corresponds to a $k$-algebra homomorphism $A\to k$,
  which is given by modding out a maximal ideal $m\subseteq A$. Let us write $\mathrm{pr}_I$ for the morphism to the quotient by an ideal $I$.
  So the formal neighborhood of $x$ in $X$ is $\Spf(\hat{A})$, where $\hat{A}$ is the completion with respect to $m$.
  This means what we need to show is, that for all $B\in\fgkAlg$, the square
  \begin{center}
    \begin{tikzcd}
      \fgkAlgTop(\hat{A}, \Delta(B))\arrow[r]\arrow[d] & \fgkAlg(k,B)\arrow[d, "\mathrm{pr}_{\sqrt{0}}\circ\blank\circ\mathrm{pr}_m"] \\
      \fgkAlg(A,B)\arrow[r,"\mathrm{pr}_{\sqrt{0}}\circ\blank", swap] & \fgkAlg(A,B/\sqrt{0})
    \end{tikzcd}
  \end{center}
  is a pullback square. This amounts to the following universal property of $\hat{A}$:

  For any $k$-algebra homomorphism $\varphi:A\to B$ such that $\varphi(m)\subseteq \sqrt{0}$, there exists a unique morphisms
  $\hat{\varphi}:\hat{A}\to B$ such that composition with the canonical $A\to\hat{A}$ is $\varphi$.
  For the construction of $\hat{\varphi}$, we may assume that $\hat{A}$ is the limit of
  \begin{center}
    \begin{tikzcd}
      \dots\arrow[r] & A/m^{n+2}\arrow[r] & A/m^{n+1}\arrow[r] & A/m^{n}
    \end{tikzcd}
  \end{center}
  for some $n\in\bN$ such that $\varphi(m^n)=\varphi(m)^n={0}$. So we have a map of sequences:
  \begin{center}
    \begin{tikzcd}
      \dots\arrow[r] & A/m^{n+2}\arrow[r]\arrow[d,"\varphi"] & A/m^{n+1}\arrow[r]\arrow[d,"\varphi"] & A/m^{n}\arrow[d,"\varphi"] \\
      \dots\arrow[r,"\mathrm{id}", swap] & B\arrow[r,"\mathrm{id}", swap] & B\arrow[r,"\mathrm{id}", swap] & B
    \end{tikzcd}
  \end{center}
  And therefore an induced $\hat{\varphi}:\hat{A}\to B$.
\end{proof}

\section{Conclusion}
During the time of writing and revising this article,
$\modal$-étale maps were already used for more calculations in real-cohesion \cite{myers2019}
and there are lots of further direction worth exploring.
Also in \cite{myers2019} the concept of $\modal$-fibrations is introduced
and used to precisely characterize the pullback squares which are preserved by a modality.
This generalizes \cref{cor:etale_lex} (i).

\printbibliography

\end{document}